\documentclass[12pt, reqno]{amsart}
\usepackage[utf8]{inputenc}

\usepackage{amssymb,latexsym}
\usepackage{amsfonts}
\usepackage{amsmath}
\usepackage{enumerate}
\usepackage{float}
\usepackage{mathtools}
\usepackage{enumerate}

\usepackage{booktabs}

\usepackage[usenames, dvipsnames]{xcolor}
\usepackage[colorlinks,linkcolor=blue,anchorcolor=blue,citecolor=blue]{hyperref}

\usepackage{graphicx}

\usepackage{caption}
\usepackage{subcaption}

\makeatletter
\@namedef{subjclassname@2020}{%
  \textup{2020} Mathematics Subject Classification}
\makeatother

\newcommand\D{{\mathbb D}}
\newcommand\pD{{\partial\mathbb D}}

\newcommand\N{{\mathbb N}}

\newcommand\BB{{\mathcal B}}
\newcommand\DD{{\mathcal D}}
\newcommand\GG{{\mathcal G}}

\newcommand\al{\alpha}
\newcommand\be{\beta}

\newcommand\ze{\zeta}
\newcommand\la{\lambda}
\newcommand\ka{\kappa}

\newcommand\var{\operatorname{var}}
\newcommand\dist{\buildrel d \over =}
\newcommand\abs[1]{\lvert#1\rvert}

\newcommand{\E}[1]{\mathbb{E}{#1}}

\newcommand{\vphi}{\varphi}
\def\leq{\leqslant}
\def\geq{\geqslant}

\newtheorem{theorem}{Theorem}

\newtheorem{corollary}[theorem]{Corollary}
\newtheorem{lemma}[theorem]{Lemma}
\newtheorem{proposition}[theorem]{Proposition}

\newtheorem{conjecture}[theorem]{Conjecture}

\theoremstyle{definition}
\newtheorem{remark}[theorem]{Remark}

\title[On $2\times2$ determinants for survival probabilities]{On $\mathbf{2\times2}$ determinants originating from survival probabilities in homogeneous discrete time risk model}

\author[A.Grigutis, J. Jankauskas]{Andrius Grigutis and Jonas Jankauskas}
\address{Institute of Mathematics, Vilnius university, Naugarduko 24, LT-03225 Vilnius, Lithuania}
\email{andrius.grigutis@mif.vu.lt}
\address{Institute of Mathematics, Vilnius university, Naugarduko 24, LT-03225 Vilnius, Lithuania}
\email{jonas.jankauskas@gmail.com}

\date{March 2021}



\begin{document}

\maketitle

\begin{abstract}
    We analyse $2\times2$ Hankel-like determinants $D_n$ that arise in search of initial values for the ultimate time survival probability $\varphi(u)=\mathbb{P}\left(\cap_{n=1}^{\infty}\left\{W(n)>0\right\}\right)$ in homogeneous discrete time risk model $W(n)=u+\kappa n-\sum_{i=1}^nZ_i$, where $Z_i$ are non-negative integer valued i.i.d. random claims, the initial surplus $u \in \N_0$ and the income rate $\kappa=2$. We prove the asymptotic version of a recent conjecture on the non--vanishing and monotonicity of $D_n$ and derive explicit formulas for the initial values $\vphi(0)$, $\vphi(1)$ of a recurrence that yields survival probabilities. In cases when $Z_i$ are Bernoulli or Geometrically distributed, the conjecture on $D_n$ is shown to hold for all $n\in\mathbb{N}_0$. Additionally, a generating function $\Xi(s)$ for ultimate survival probabilities $\vphi(u)$ is derived.      
\end{abstract}

\maketitle

\section{Introduction}\label{sec:intro}

Let $Z_i$, $1 \leqslant i \leqslant n$ be i.i.d. copies of a discrete random variable (r.v.) $Z$ that takes only non-negative integer values. The sum $\sum_{i=1}^{n}Z_i$ is called {\it the random walk} (\emph{r.w.}). Random walks appear on various occasions in many fields of mathematics, both pure and applied. For instance, in finance and insurance, the accumulated wealth in discrete moments of time $n \in \{0, 1, 2 \dots\}=:\N_0$, can be modeled by formulas
\begin{align}\label{eq:model}
W(0):=u,\,W(n)=u+\kappa n-\sum_{i=1}^{n}Z_i,\,n\geqslant1,
\end{align}
where the parameter $u\in \mathbb{N}_0$ is called \emph{the initial surplus}, $\kappa\in\mathbb{N}$ is \emph{the income rate} and $Z_i$ represent randomly occurring \emph{claims}. The model defined in \eqref{eq:model} is a simplified discrete version of more general Sparre Andersen model \cite{Andersen}. In such a form as in (\ref{eq:model}) it was introduced and studied in \cite{GS}. 

The main concern on the accumulated wealth in (\ref{eq:model}) is whether $W(n)>0$ for every $n$. In other words, whether the initial savings and the subsequent income is always sufficient to cover incurred expenses. One desires to know the probability for the r.w. to hit the line $u+\kappa n$ at least once up to some natural $n$. To answer that for a finite $n$ is just a simple probabilistic problem, see \cite[Theorem 1]{GS}. However, the qualitative break appears as $n\to\infty$. For this, let us define
\begin{align}\label{eq:surv}
\varphi(u):=\mathbb{P}\left(\bigcap_{n=1}^{\infty}\left\{W(n) > 0\right\}\right).
\end{align}
The function $\varphi(u)$ is called {\it the ultimate time survival probability}. Using the law of total probability and elementary rearrangements (see \cite[p. 3]{GS}) it can be shown that
\begin{align}\label{eq:main}
\varphi(u)=\sum_{n=1}^{u+\kappa}h_{u+\kappa-n}\varphi(n),
\end{align}
where $h_k=\mathbb{P}(Z=k)$ for all $k\in\mathbb{N}_0$. We assume that $h_0>0$, for otherwise one could replace every $Z_i$ in (\ref{eq:model}) with $Z_i-1$, and $\kappa$ with $\kappa -1$, which causes the reduction in order of the recurrence, see \cite[Theorems 3 and 4]{GS}.

The recursive nature of the equation \eqref{eq:main} makes it very appealing for the numerical computation of the survival probabilities $\vphi(u)$. To avoid getting entangled in too many details, for the demonstration let us switch to $\kappa=2$. For the moment, assume that the initial values  $\vphi(0), \vphi(1)$ are known beforehand. Then, the remaining values $\varphi(u)$, $u \geq 2$ can be solved for, by setting $u=0,\, 1,\,\ldots ,\,n$ in (\ref{eq:main}) as follows:
\begin{flalign*}
\varphi(0)=h_1\varphi(1)+h_0\varphi(2) \implies \varphi(2)=\frac{1}{h_0}\varphi(0)-\frac{h_1}{h_0}\varphi(1),
\end{flalign*}
\begin{flalign*}
\begin{split}
(1-h_2)\varphi(1)-h_1\varphi(2)-h_0\varphi(3)&=0 \implies\\ \varphi(3) & =\frac{-h_1}{h_0^2}\varphi(0)+\frac{h_0+h_1^2+h_0h_2}{h_0^2}\varphi(1).
\end{split}
\end{flalign*}
Using mathematical induction the following equalities can be derived (for details see \cite[p. 17]{GS}),
\begin{align}\label{n via 01}
\varphi(n)=x_n\varphi(0)+y_n\varphi(1),
\end{align}
where the deterministic sequences $x_n$ and $y_n$ are given by
\begin{equation}\label{def:xn}
x_0 := 1,\, x_1 := 0, \, x_n := \frac{1}{h_0} \left( x_{n-2} -\sum_{i=1}^{n-1} h_{n-i} x_i\right), \text{ for } n \geqslant 2
\end{equation} 
and
\begin{equation}\label{def:yn}
y_0 := 0,\,y_1 := 1,\,  y_n := \frac{1}{h_0} \left( y_{n-2} -\sum_{i=1}^{n-1} h_{n-i} y_i\right), \text{ for } n \geqslant 2.
\end{equation}

A standard way to obtain the initial values $\varphi(0)$, $\varphi(1)$ is an application of stationarity and total expectation for distribution of maximum of sums $\sum_{i=1}^n (Z_i - \kappa)$. Such a derivation is outlined in Section \ref{sec:alt}. However, this standard method requires the explicit knowledge on the poles of the probability generating function (p.g.f.). In real life problems, the precise computation of poles might not be very efficient or even impossible in cases where the distribution is not given analytically and must be estimated from the limited number of observations (especially for heavy tailed distributions). However, as far as \emph{the numerical approximation} of initial values $\vphi(0)$, $\vphi(1)$ is concerned, there exists a neat procedure where one does not need to solve for the poles of the p.g.f. explicitly. It works as follows. For every $n\in\mathbb{N}_0$, \eqref{n via 01} implies 
\begin{align}\label{system}
\left( \begin{array}{cc}
{x}_{n} & y_n  \\
x_{n+1} & y_{n+1}
\end{array} \right)
\times
\left( \begin{array}{c}
\varphi(0) \\
\varphi(1)
\end{array} \right)
=
\left( \begin{array}{c}
\varphi(n) \\
\varphi(n+1)
\end{array} \right).
\end{align}

Let
\begin{equation}\label{def:Dn}
    D_n:= \begin{vmatrix}
            x_n & y_n\\
            x_{n+1} & y_{n+1}\\
        \end{vmatrix},\, n\in\mathbb{N}_0
\end{equation}
be the principal determinant of \eqref{system}. If $D_n \ne 0$, then $\varphi(0)$ and $\varphi(1)$ can be solved from the system \eqref{system}:
\begin{equation}\label{eq:limits}
\varphi(0) := \frac{y_{n+1}}{D_n}\varphi(n) - \frac{y_n}{D_n}\varphi(n+1), \, \varphi(1) := \frac{x_n}{D_n}\varphi(n+1) - \frac{x_{n+1}}{D_n}\varphi(n)
\end{equation}
Obviously, $\varphi(n)\leqslant\varphi(n+1)\leqslant1$ for all $n\in\mathbb{N}_0$. This implies that the limit $\varphi(\infty):=\lim_{n \to \infty}\varphi(n) \in [0, 1]$ exists. In addition, if the limits of the ratios
\begin{equation}\label{eq:limrat}
\frac{x_n}{D_n}, \quad \frac{x_{n+1}}{D_n}, \quad \frac{y_n}{D_n}, \quad \frac{ y_{n+1}}{D_n},
\end{equation}
exist, then
\begin{align}\label{eq:ivp}
\varphi(0) = \varphi(\infty)\lim_{n \to \infty}\frac{y_{n+1} - y_n}{D_n},
\,\varphi(1) = \varphi(\infty)\lim_{n \to \infty}\frac{x_n - x_{n+1}}{D_n}.
\end{align}
Using the law of large numbers it can be proved \cite[Lemma 1]{GS} that $\varphi(\infty)=1$ if the expectation $\mathbb{E}Z<2$. If $\mathbb{E}Z \geqslant 2$, then $\varphi(n)$, $n \in \N_0$ and $\varphi(\infty)$ all vanish \cite[Theorem 9]{GS}. Intuitively, this means that the survival is possible if claims, represented by $Z$, are not too "aggressive" on average. Thus, it remains to determine the conditions under which the limits of the ratios \eqref{eq:limrat} exist and can be used in \eqref{eq:limits}.

For the reasons discussed above, the determinants $D_n$ are the main objects of interest in the present paper.  The non-vanishing of $D_n$ for each $n \in \N$, and the asymptotic value for large $n$  is considered in the context of the numerical reconstruction of the survival probabilities \eqref{eq:surv}. In \cite[p. 6]{GS}, numerical computations with some selected distributions of $Z$ led to the following conjecture.

\begin{conjecture}\label{conj:Dn}
For every $n \in \N_0$, 
\[
1 \leqslant D_{2n} \leqslant D_{2n+2}
\qquad \text{ and } \qquad D_{2n+3} \leqslant D_{2n+1} \leqslant -1.
\]
\end{conjecture}

So far, numerical calculations did not reveal any counterexamples of Conjecture \ref{conj:Dn}. It is worth to mention that different generalizations, comparing to the model in \eqref{eq:model}, covering different r.w. setups, time, income rates and etc. are known: one can refer to \cite{Andersen}, \cite{DS}, \cite{Dickson_Waters}, \cite{Gerber}, \cite{Gerber1}, \cite{GKS}, \cite{GS1}, \cite{Shiu}, and many other papers. In this paper we show that in cases where $\mathbb{P}{(Z \in 2\N_0+1)} = 0$, Conjecture \ref{conj:Dn} admits almost a trivial proof which is given in Section \ref{sec:genfun}, Proposition \ref{lemma:imprDn}. When $\mathbb{P}{(Z \in 2\N_0+1)} \ne 0$, no proof is available yet. In Section \ref{sec:examples}, we verify the conjecture in cases when $Z$ is Bernoulli or Geometric r.v. - Theorem \ref{thm:bern_geom} in Section \ref{sec:results}.

In the present paper we prove an asymptotic version of Conjecture \ref{conj:Dn} that requires the finiteness of higher moments of $Z$ but does not depend on the specific distribution - Theorem \ref{thm:main_res} in Section \ref{sec:results}.

More precisely, in Section \ref{sec:asymp_exp}, by Corollary \ref{cor:cfs}, Theorem \ref{thm:asymtx} and Theorem \ref{thm:asymtD_n} we derive an exact expressions for dominant terms of $x_n$, $y_n$ and $D_n$ in \eqref{eq:ivp} which yield the explicit initial values $\vphi(0)$ and $\vphi(1)$ for recurrence \eqref{eq:main} - Corollary \ref{cor:ivp} in Section \ref{sec:results}. Moreover, we give a closed-form expression of  generating function of $\vphi(u+1),\,u\in\mathbb{N}_0$ - Theorem \ref{thm:Phi} in Section \ref{sec:results}.

It should be noted that proving monotonicity or even non-vanishing of $D_n$ for all $0\leqslant n \leqslant n_0$ seems to be a non-trivial problem in all but few trivial cases. For instance, when the p.g.f. of $Z$ is a rational function (the ratio of two polynomials with real coefficients), then to solve $D_n=0$ for $n$ is equivalent to finding all zeros of a certain linear recurrence. This includes even the very basic case when the r.v. $Z$ has a finite support (its p.g.f. is polynomial). There are no general explicit analytic formulas for finding such $n$. Typically, the solution involves obtaining numerical upper bound $n_0$ for $n\in\mathbb{N}$ or the total number of such possible solutions using some sophisticated number--theoretical machinery, and then checking the range of possible $n\in\mathbb{N}$ with computers. Monograph \cite{Graham} is an excellent source of references on this and related topics.

As the initial values of survival probabilities depend on the location of zeros of a p.g.f. of the r.v. $Z$, it is unlikely that there exists one simple formula that covers all possible cases for arbitrary $\kappa \in \N$. To avoid being buried by extensive technical details and the large number of cases to work through, in the present paper we deal with the most simple version of Conjecture \ref{conj:Dn} for $\kappa=2$. For this particular $\kappa$ , Conjecture \ref{conj:Dn} is very accessible to our analysis. As we prepare mathematical tools to extend these results beyond $\kappa\geq 3$ or for models with more complicated r.w. setups in the future, some of our lemmas in Section \ref{sec:zero_locs} are stated in a more general form for arbitrary $\kappa \in \N$.

\section{Main results}\label{sec:results}

In this section we formulate main results of the paper. As mentioned in introduction, Conjecture \ref{conj:Dn} is correct if $Z$ is Bernoulli or Geometric distributed. 

\begin{theorem}\label{thm:bern_geom}
Let $1-\mathbb{P}(Z=0)=p=\mathbb{P}(Z=1)$ or $\mathbb{P}(Z=k)=p(1-p)^k,\,k\in\mathbb{N}_0$, where $p \in (0, 1)$. For such r.v. $Z$, Conjecture \ref{conj:Dn} is true.
\end{theorem}
We prove Theorem \ref{thm:bern_geom} in Section \ref{sec:examples}. As noted, Conjecture \ref{conj:Dn} admits almost a trivial proof if $\mathbb{P}({Z} \in 2\N_0+1)=0$ - see Proposition \ref{lemma:imprDn} in Section \ref{sec:genfun}. However, when $\mathbb{P}({Z} \in 2\N_0+1)\neq0$, Conjecture \ref{conj:Dn} is still open for all $n\in\mathbb{N}$, while for $n\to\infty$, the following statement is correct.

\begin{theorem}\label{thm:main_res}
Assume that $\mathbb{P}({Z} \in 2\N_0+1) \ne 0$, and $\E{Z} < +\infty$. Furthermore, suppose that the  higher moments of $Z$ satisfy:
\[
    \E{Z^2} < +\infty, \text{ if }\; \E{Z} \ne 2 \quad \text{ or } \quad \E{Z^4} < +\infty, \text{ if } \; \E{Z} = 2.
\]
Then,
$$
1 < D_{2n} < D_{2n+2}, \quad D_{2n+3} < D_{2n+1} < -1,
$$
for every $n > n_0$, where $n_0$ depends on $Z$ only, and
$$
D_{2n} \to +\infty, \qquad D_{2n+1} \to -\infty,
$$
as $n \to +\infty$.
\end{theorem}

We prove Theorem \ref{thm:main_res} in Section \ref{sec:asymp_exp}. Theorem \ref{thm:main_res} leads to the following expressions of the initial values $\vphi(0)$ and $\vphi(1)$ for the recurrence relation \eqref{eq:main}.

\begin{corollary}\label{cor:ivp}
Let $h_0=\mathbb{P}{(Z=0)}>0$ and $\E{Z} < 2$.

If $\mathbb{P}({Z} \in 2\N_0+1) = 0$, then
\[
\vphi(0) = \frac{2-\mathbb{E}Z}{2},\, \qquad \vphi(1) =\frac{2-\mathbb{E}Z}{2h_0}.
\]

If $\mathbb{P}({Z} \in 2\N_0+1) \ne 0$, and, in addition, $\E{Z}^2 < +\infty$, then
\[
\vphi(0) = \lim_{n \to \infty}\frac{y_{n+1} - y_n}{D_n} = \frac{\al(2-\E{Z})}{1+\al},
\]
\[
\vphi(1) = \lim_{n \to \infty}\frac{x_{n} - x_{n+1}}{D_n} =\frac{2-\E{Z}}{h_0(1+\al)},
\]
where $-\al^{-1}\in(-1,0)$ denotes the unique real solution of the equation $H(s)=s^2,\,s\in \mathbb{C}$ and $H(s)$ is the probability generating function (p.g.f.) of r.v. $Z$.
\end{corollary}

On the other hand, the statement of Corollary \ref{cor:ivp} can be derived differently - see Section \ref{sec:alt}. Moreover, arguments given in Section \ref{sec:alt} allows to set up the generating function of $\vphi(u+1),\,u\in\mathbb{N}_0$.

\begin{theorem}\label{thm:Phi}
The generating function $\Xi(s):=\sum_{n=0}^{\infty}\vphi(n+1)s^n$ of the ultimate time survival probability satisfies
\[
\Xi(s) = \frac{(2-\E{Z})^+(1+\al s)}{(1+\al)(H(s)-s^2)},
\]
where $\al$ and $H(s)$ are the same as in Corollary \ref{cor:ivp}, and, for $a \in \mathbb{R}$, $a^+ = \max\{0, a\}$ is the positive part function.
\end{theorem}
We derive the statement of Theorem \ref{thm:Phi} in Section \ref{sec:alt} too. 

\section{Generating functions}\label{sec:genfun}

Recall that the probabilities $h_n = \mathbb{P}(Z=n) \geqslant 0$ for $n\in\mathbb{N}_0$ satisfy  $\sum_{n\geqslant 0} h_n = 1$. Here we require $h_0 > 0$. The probability generating function $H(s)$ of the r.v. $Z$ is defined as the power series

$$
H(s) := \sum_{n \geqslant 0} h_n s^n,\, s\in\mathbb{C}.
$$

We call an arbitrary power series \emph{imprimitive}, if there exists an integer $d \geqslant 2$, such that $d$ divides $n$ (denoted $d\mid n$) whenever $s^n$ is present in the series. In particular, $H(s)$ is imprimitive when $h_n \neq 0$ only for $n$ divisible by $d$. In that case, one can re-write $H(s)$ as $H(s)=H_1(s^d)$ for a p.g.f. $H_1(s)$ of the r.v. $Z/d$. If no such $d \geqslant 2$ exists, then we call $H(s)$ \emph{primitive}.

For $\kappa=2$, the fact that the power series $H(s)-s^2$ is primitive means $h_{2n+1} \ne 0$ for at least one $n \in \N_0$; that is, $\mathbb{P}(Z\in 2\N_0 +1) \ne 0$. If $H(s)$ is not primitive, it means that $\mathbb{P}(Z \in 2\N_0)=1$. If such situation arises, one could consider the process $W_1(n)$, obtained by replacing each $Z_i$ with $Z_i/2$ and $\kappa=2$ with $\kappa/2=1$ in Eq. \eqref{eq:model}. By denoting the ultimate time survival probability function of $W_1(n)$ by $\vphi_1(u)$, one can show that, in the imprimitive case, $\vphi(2u)=\vphi(2u-1)=\vphi_1(u)$. So, it would be sufficient to consider the survival probabilities of the process $W_1(n)$ with the reduced income rate $\kappa=1$ instead of $W(n)$ with $\kappa=2$.

We also use the notations
\[
\D = \left\{s \in \mathbb{C}: \abs{s} < 1 \right\},\, \partial\D = \left\{s \in \mathbb{C}: \abs{s} = 1 \right\},\,
\overline{\D} = \left\{s \in \mathbb{C}: \abs{s} \leqslant 1 \right\}
\]
for the open unit disk and the unit circle in a complex plane $\mathbb{C}$.

The generating functions of sequences $x_n$, $y_n$ $n \in \N_0$ from \eqref{def:xn}, \eqref{def:yn} for $s\in\mathbb{C}$ are defined by
\begin{equation}\label{def:X}
    X(s) := \sum_{n \geqslant 0} x_ns^n, \qquad Y(s) := \sum_{n \geqslant 0} y_ns^n.
\end{equation}
From \eqref{def:xn} we have
\[
x_{n-2}=\sum_{i=0}^{n}x_i h_{n-i}-x_0h_n, \, n\geqslant 2.
\]
From this,
\[
\sum_{n\geqslant2}x_{n-2} s^n=\sum_{n\geqslant2}\left(\sum_{i=0}^{n}x_i h_{n-i}\right)s^n-x_0\sum_{n\geqslant2}h_n s^n
\]
or
\[
s^2X(s) = X(s)H(s) -x_0h_0 - (x_1h_0+x_0h_1)s - x_0(H(s)-h_0-h_1s),
\]
which simplifies to
\begin{equation}\label{eq:genX}
s^2X(s) = X(s)H(s) - x_1h_0s - x_0H(s).    
\end{equation}
From initial conditions $x_0=1$, $x_1=0$, one obtains
\begin{equation}\label{eq:X}
X(s) = \frac{H(s)}{H(s)-s^2}.
\end{equation}
By replacing $x_n$ with $y_n$ and using the appropriate initial conditions $y_0=0$, $y_1=1$ in \eqref{eq:genX}, one obtains
\begin{equation}\label{eq:Y}
Y(s) = \frac{h_0s}{H(s)-s^2}.
\end{equation}
Thus, from (\ref{eq:X}) and (\ref{eq:Y})
\begin{equation}\label{eq:eliminate}
h_0X(s)-sY(s) = h_0.
\end{equation}

The power series expansion of (\ref{eq:eliminate}) yields
\begin{equation}\label{eq:yfromx}
h_0x_n - y_{n-1} =0,\, n \geqslant 1 \quad \text{ or } \quad y_n = h_0x_{n+1},\, n \geqslant 0.
\end{equation}

Then, $D_n$ in \eqref{def:Dn} becomes
\begin{equation}\label{eq:Dx}
     D_n = \begin{vmatrix}
            x_n & y_n\\
            x_{n+1} & y_{n+1}\\
        \end{vmatrix} =
         \begin{vmatrix}
            x_n& h_0 x_{n+1}\\
            x_{n+1}& h_0 x_{n+2}\\
        \end{vmatrix} = h_0 \left(x_n x_{n+2} - x_{n+1}^2\right).
\end{equation}

Thus, $D_n$ is $h_0$ multiple of Hankel determinant of the second order, see \cite[Chapter 10]{Gantmacher}. 

\begin{proposition}\label{prop:xn_mon}
For every $n\in \mathbb{N}_0$,
\[
1\leqslant x_{2n} \leqslant x_{2n+2} \qquad \text{ and } \qquad  x_{2n+3} \leqslant x_{2n+1} \leqslant 0.
\]
\end{proposition}
\begin{proof}[Proof of Proposition \ref{prop:xn_mon}.]
For $n=0$, $x_0=1\leqslant1/h_0=x_2$ and $x_1=0\geqslant-h_1/h^2_0=x_3$. By induction,
\begin{align*}
&x_{2n+2}=\frac{1}{h_0}\left(x_{2n}-\sum_{i=1}^{2n+1}h_{2n+2-i}x_i
\right)=\\
&=\frac{1}{h_0}\left(x_{2n}-(h_{2n+1}x_1+\ldots+h_1x_{2n+1})
-(h_{2n}x_2+\ldots+h_2x_{2n})\right) \\
&\geqslant
\frac{1}{h_0}\left(x_{2n}-x_{2n}(h_1+h_2+\ldots)\right)=x_{2n}
\end{align*}
and
\begin{align*}
&x_{2n+3}=\frac{1}{h_0}\left(x_{2n+1}-\sum_{i=1}^{2n+2}h_{2n+3-i}x_i
\right)\\
&=\frac{1}{h_0}\left(x_{2n+1}-(h_{2n+2}x_1+\ldots+h_2x_{2n+1})
-(h_{2n+1}x_2+\ldots+h_1x_{2n+2})\right)\\
&\leqslant
\frac{1}{h_0}\left(x_{2n+1}-x_{2n+1}(h_1+h_2+\ldots)\right)=x_{2n+1}.
\end{align*}
Similar inequalities were obtained in \cite{DS} for a different model than defined in (\ref{eq:model}).
\end{proof}
It is curious that the monotonicity property from Proposition \ref{prop:xn_mon} is sufficient to establish Conjecture \ref{conj:Dn} when $H(s)-s^2$ is imprimitive.
\begin{proposition}\label{lemma:imprDn}
If $\mathbb{P}(Z\in 2\N_0+1) = 0$, then Conjecture \ref{conj:Dn} is true.
\end{proposition}
\begin{proof}[Proof of Proposition \ref{lemma:imprDn}.]
The equality $h_{2n+1}=0$ yields $H(s)=H_1(s^2)$ for a p.g.f. $H_1(s)$. By \eqref{eq:X}, we have $X(s)=X_1(s^2)$ for $X_1(s)=H_1(s)/(H_1(s)-s)$. It follows that $X(s)$ is a power series of $s^2$, which implies $x_{2n+1}=0$ for all $n\in\mathbb{N}_0$. Then, by \eqref{eq:Dx},
\[
D_{2n} = h_0x_{2n}x_{2n+2}, \qquad D_{2n+1}=-h_0x_{2n+2}^2. 
\]
As $x_0=1$ and $x_{2n}$ is non--decreasing by Proposition \ref{prop:xn_mon}, the sequence $D_n$ has the required properties.
\end{proof}
However, such a simple trick is not sufficient to prove Conjecture \ref{conj:Dn} in the primitive case.

\section{Location and properties of zeros}\label{sec:zero_locs}

Following \cite[Ch.VII, Sec.5]{Feller} in the technique of analysis, we prove a series of technical lemmas about the location of zeros and the vanishing multiplicity of the power series of the form $H(s)-s^\ka$, $s \in \mathbb{C}$ and $\kappa \in \N$, where $H(s)$ is the p.g.f. of the r.v. $Z$. The power series of this form $H(s)-s^\ka$ appears in the denominators of generating functions $X(s)=H(s)/(H(s)-s^\kappa)$ for the corresponding recurrence \eqref{def:xn} with arbitrary natural $\ka$ in \eqref{eq:main}. The choice $\kappa=2$ corresponds the generating functions derived in Section \ref{sec:genfun} and Conjecture \ref{conj:Dn}. We allow arbitrary $\kappa \in \N$ in this section for the future references to the lemmas presented here. However, to single out the main result, which is used in this work from this auxiliary section, we would like to highlight Corollary \ref{cor:m2zerosH}. It provides location and multiplicity of roots of $H(s)-s^2$. 

Let us now recall the very basic properties of p.g.f. $H(s)$.

\begin{lemma}\label{prop:Hprops}
The function $H(s)$ is holomorphic in $\D$ and continuous on its boundary $\pD$. In addition, if $H^{(k)}(1) < +\infty$, $k \in \N$, then the derivatives $H^{(j)}(s)$, $0 \leqslant j \leqslant k$ are continuous on $\overline{\D}$.  
\end{lemma}
\begin{proof}[Proof of Lemma \ref{prop:Hprops}.]
For $\abs{s} \leqslant 1$, $\sum_{n \geqslant 0}\abs{h_ns^n} \leqslant \sum_{n \geqslant 0}h_n =1$, so the convergence of the power series $H(s)$ is absolute and uniform. It follows that $H(s)$ is holomorphic inside the unit disk and continuous on its boundary.

Similarly, the sum of absolute values of the terms in the series $H^{(k)}(s)$ in $\overline{\D}$ is less or equal to $ \sum_{n \geqslant 1}n!h_n/(n-k)! = H^{(k)}(1) < +\infty$. Thus, $H^{(k)}(s)$ converges uniformly to a continuous function for $\abs{s} \leqslant 1$. This implies that the derivatives of $H(s)$ of order $\leqslant k$ are well defined and continuous on $\partial\D$.  
\end{proof}

\begin{lemma}\label{lem:zerosH}
The function $H(s)-s^\ka$ has at most $\ka$ zeros in $\D$, counted with their multiplicities.
\end{lemma}
\begin{proof}[Proof of Lemma \ref{lem:zerosH}]
For every $s \in \pD$, and every real $\la > 1$, $\abs{H(s)} \leqslant 1 < \abs{\la s^\ka}$. Hence, by Rouch\'e's theorem \cite[Ch.10, Ex.24]{Rudin}, $H(s)-\la s^\ka$ has the same number of zeros in $\D$ as $s^\ka$. By continuity of zeros inside $\D$ with respect to the parameter $\la$, as $\la \to 1^+$, the number of zeros cannot increase when $\la$ reaches $1$ (it can only decrease, if some zeros from $\D$ reach the boundary $\pD$ at $\la=1$).
\end{proof}

In subsequent lemmas, the positive integer $d$ is not restricted to $d \geq 2$ as before (it can also equal to $1$).

\begin{lemma}\label{lem:value1H}
The equality $H(s)= \pm s^\ka$ holds on $\partial \D$ only at points $s$ that satisfy $s^d =1$ for $d \in \N$, such that $d \mid n$ whenever $h_n \ne 0$ and $d \mid \ka$ in '$+s^\ka$' case, $d \mid 2\ka$ in '$-s^\ka$' case. In particular, if $H(s)-s^\ka$ is primitive, then $H(s)=s^\ka$ holds on $\pD$ only at $s=1$.
\end{lemma}
\begin{proof}[Proof of Lemma \ref{lem:value1H}.]
In the triangle inequality,
\[
1 = \abs{H(s)s^{-\ka}} \leqslant \sum_{n \geqslant 0}h_n\abs{s^{n-\ka}},
\]
the equality "$=$" is attained only when all non-zero terms have the same complex argument. This implies that $s^{l-k}$ is real positive whenever $h_lh_k \ne 0$. Equality $\abs{s}=1$ implies $s^{l-k}=1$ and all such $s \in \D$ must be roots of unity whose orders divide all the differences $l-k$ for $h_kh_l \ne 0$. Since $h_0 \ne 0$, for such a root of unity of the minimal order $d$, it follows that all $h_l \neq 0$ must lie in some arithmetic progression $l=dj$, $j \in \N$. Thus, one can write $H(s)=H_1(s^d)$. Then, for such a root of unity, $H_1(s^d)=1$ and $H(s)=s^\ka$ imply $s^\ka=1$, and $d \mid \ka$, or $s^\ka=-1$, which means $d \mid 2\ka$. The primitive case now becomes obvious.
\end{proof}
\begin{lemma}\label{lem:ord1}
    Let $r$ denote the order of vanishing of $H(s)-s^\ka$ at $s=1$. If $\mathbb{P}(Z=\ka)<1$, then $r \leqslant 2$.
\end{lemma}
\begin{proof}[Proof of Lemma \ref{lem:ord1}]
    For the real $s \in [0, 1]$, the repeated application of Cauchy Middle Value Theorem for $(H(s)-s^\ka)/(s-1)^k$, when $k\leq r$, implies that finite one sided limits
    \begin{equation}\label{eq:lim1}
        \lim_{s \to 1^{-}}\frac{H(s)-s^\ka}{(s-1)^k} = \begin{cases}
            \frac{H^{(k)}(1)-\ka!/(\ka-k)!}{k!}, & \text{ for } k \leqslant \ka,\\
            \frac{H^{(k)}(1)}{k!}, & \text{ for } k > \ka,
        \end{cases}
    \end{equation}
    exist for $k \leqslant r$. Limits in (\ref{eq:lim1}) must equal $0$ for $k < r$ and, for $k=r$, limit must be finite and non--zero. It follows that $H^{(r)}(1)$ exist and are $< +\infty$. By Lemma \ref{prop:Hprops}, derivatives $H^{(j)}(s)$, $0 \leqslant j \leqslant r$ are continuous on $\overline{\D}$.
    
    Assume that $r \geq 3$. Then, one has $H'(1)=\ka$, $H''(1)=\ka(\ka-1)$, $H'''(1) \not\in \{0, +\infty\}$ by \eqref{eq:lim1}. On the other hand, $H'(1)=\E{Z}$, $H''(1)=\E{Z^2}-\E{Z}$ which implies $\E{Z}=\kappa$, $\E{Z^2}=\kappa^2$. Then, the variance $\var{Z}=0$ and consequently r.v. $Z$ is degenerated: $\mathbb{P}(Z=\ka)=1$, $H(s)=s^\ka$. Therefore, the limit in \eqref{eq:lim1} is zero for all $k$ and the assumption $r\geq3$ is contradicted.
\end{proof}
\begin{lemma}\label{lem:divideH}
For the power series $H(s)$ it holds that $H(s) - s^\ka = (1-s)G(s)$, where $G(s)=\sum_{n\geqslant0}g_n s^n$, and the coefficients
\[
g_n = \begin{cases}
        \sum_{k \leqslant n} h_k, & \text{ if } n < \ka,\\
        -1+\sum_{k \leqslant n} h_k, & \text{ for } n \geqslant \ka.
    \end{cases}
\]
Moreover:
\begin{itemize}
\item If $H'(1) > \ka$, then $G(s)$ has one simple real zero $s\in(0, 1)$, and $G(1)\ne 0$.

\item If $H'(1) = \ka$, then $G(s)$ vanishes for $s\in[0, 1]$ only at $s=1$.

\item If $H'(1) < \ka$, then $G(s) > 0$ for all $s\in[0, 1]$.
\end{itemize}

\end{lemma}
\begin{proof}[Proof of Lemma \ref{lem:divideH}.]
As $\sum_{n\geqslant0} h_n = 1$,
\[
H(s) - s^\ka = \sum_{n\geqslant 0} h_n(s^n - s^\ka)= \sum_{n < \ka} h_n(s^n - s^\ka) - \sum_{n > \ka} h_n(s^\ka - s^n),
\]
therefore
\[
G(s) = \frac{H(s) - s^\ka}{1-s}=\sum_{n < \ka} h_n s^n\frac{1-s^{\ka-n}}{1-s} - \sum_{n > \ka} h_n s^\ka\frac{1-s^{n-\ka}}{1-s}
\]
\[
= \sum_{n < \ka} h_n \left(\sum_{k=n}^{\ka-1}s^k\right) - \sum_{n > \ka} h_n\left(\sum_{k=\ka}^{n-1}s^k\right)
\]
\[
= \sum_{k < \ka} s^k \left(\sum_{n \leqslant k}h_n\right) - \sum_{k \geqslant \ka} s^k\left(\sum_{n>k}h_n\right)
\]
\[
= \sum_{k < \ka} s^k \left(\sum_{n \leqslant k}h_n\right) - \sum_{k \geq \ka} s^k\left(1 - \sum_{n \leqslant k}h_n\right).
\]
Interchanging $k$ and $n$ yields the above claimed formulas.

As the coefficients $g_n\geqslant$0, for $n < \ka$, and $g_n \leqslant 0$, for $n \geqslant \ka$, $G(s)$ and $G'(s)$ have at most one sign change each; by Descartes rule of signs for power series \cite{Curt}, it follows that each of $G(s)$  and $G'(s)$ can have at most $1$ simple positive real zero in $(0, 1]$. As $G(0)=h_0>0$, $G(s)$ must have one simple zero in $(0, 1]$ if $G(1) < 0$. If $G(1) > 0$, then $G(s)$ does not vanish in $[0, 1]$, because in such case it must vanish twice, or have a zero of even multiplicity, which would contradict the aforementioned Descartes rule. It remains to consider the possibility that $G(1)=0$.  If $G(s)$ vanishes at some other point $a \in (0, 1)$, then $G'(s)$ must change its sign between $0$ and $a$: indeed, as $G(0)=h_0>0$, $G'(0)=h_0+h_1 >0$ (for $\ka \geqslant 2$), $G'(s)$ must become negative in $(0, a)$ for $G(s)$ to descend to $0$ at $s=a$. Then, by Rolle's theorem, $G'(s)$ has at least two zeros: one in the interval $(0, a)$, as it was discussed above, and another in $(a, 1)$, contradicting the sign rule applied to $G'(s)$ \cite{Curt}. The same is also true when $\ka=1$: $G'(s) < 0$ for $s \in (0, 1]$ or $G(s)$ is a constant (because $g_n\leqslant 0$, for $n \geqslant 1$, when $\kappa=1$). Therefore $G(s)$ can have only one zero in $[0, 1]$, when $\ka=1$. One evaluates $G(1)$ by
\[
G(1) = \lim_{s \to 1^-} \frac{H(s)-s^\ka}{1-s} = (s^\ka - H(s))'\vert_{s=1} = \ka - H'(1).
\]
This proves all the properties of $G(s)$ claimed in Lemma \ref{lem:ord1}.
\end{proof}

\begin{lemma}\label{lem:rhoH} For any complex zero $\zeta \in \D$  of $H(s)-s^\ka$, it's absolute value $\abs{\zeta}$ belongs to the interval $(0, a]$; here $a \in (0, 1]$ denotes the smallest positive zero of $H(s)$. Moreover, $\abs{\zeta}=a$ is possible only when $\zeta=\abs{a}e^{2\pi {\bf i}j/d}$, ${\bf i}:=\sqrt{-1}$, for $0 \leqslant j \leqslant d-1$, where the integer $d\mid\ka$ and $d\mid n$ for each $n$ such that $h_n \ne 0$.
\end{lemma}
\begin{proof}
By the previous Lemma \ref{lem:divideH}, $H(s)-s^\ka$ has unique positive simple real zero $a \in (0, 1]$, such that $H(s) > s^\ka$ for $s \in [0, a)$, and $H(s) < s^\ka$ for $s \in (a, 1)$ (if $a=1$, then the later interval is empty). After taking absolute values on both sides of $H(\ze)=\ze^\ka$, one obtains
\begin{equation}\label{eq:moduli2}
\sum_{n \geq 0}h_n \abs{\zeta^n} \geqslant \abs{\zeta}^\ka,
\end{equation}
which is equivalent to $H(\abs{\ze}) - \abs{\ze}^\ka \geqslant 0$. Thus, $\abs{\ze} \in (0, a]$. Furthermore, $\abs{\ze}=a$ is possible for $H(\ze)=\ze^\ka$ only when equality is attained in the triangle inequality in \eqref{eq:moduli2}. Reasoning the same way as in Lemma \ref{lem:value1H}, all non-zero terms $h_n\ze^n$ and $\ze^\ka$ must be real and positive, and there exists such smallest $d \in \N$, which $d 
\mid \ka$, $d \mid n$ whenever $h_n \ne 0$ and $\ze^d = a^d$. The statement follows.
\end{proof}
\begin{corollary}\label{cor:1val1H}
Let $a \in (0, 1]$ denote the smallest positive zero of $H(s)-s^\ka$. If $\ka$ is even and $h_n \ne 0$ for at least one odd $n \in \N$, then $\abs{H(-1)} < 1$ and $H(s)-s^\ka$ has odd number of negative real zeros in $(-1, 0)$, each of them of odd order and located in $(-a, 0)$.
\end{corollary}
\begin{proof}[Proof of Corollary \ref{cor:1val1H}.]
As $0<h_0<1$, one readily verifies that $H(-1) = \pm 1$ is impossible, if $h_n \ne 0$ for at least one odd $n$ (see also Lemma \ref{lem:value1H}). Hence, $\abs{H(-1)} < 1$. It follows that $H(s)-s^\ka< 0$ at the point $s=-1$ and $H(s)-s^\ka=h_0 > 0$ at $s=0$, so the function $H(s)-s^\ka$ must have an odd number of sign change points in $(-1, 0)$. Furthermore, by Lemma \ref{lem:rhoH}, $\abs{s} \leq a$. As $s=-a$ is possible only for even integers $d$ in Lemma \ref{lem:value1H} (for $a=1$) or Lemma \ref{lem:rhoH} (for $a<1$), we must have $\abs{s} < a$ for each such sign change point $s \in (-a, 0)$.
\end{proof}

\begin{corollary}\label{cor:m2zerosH}
Assume that $H(s)-s^2$ is primitive. Then $H(s)-s^2$ has at most $2$ simple, distinct zeros inside $\D$ and one zero on the boundary $\pD$ of multiplicity at most $2$. More precisely, $H(s)-s^2$ has
\begin{enumerate}[a)]
\item a simple negative zero at $s=-\al^{-1} \in (-1, 0)$.
\item a simple positive zero at $s=\be^{-1} \in (\al^{-1}, 1)$, when $H'(1) > 2$.
\item a zero at $s=1$. If $H'(s) \ne 2$, then this zero is simple. If $H'(1)=2$, $H''(1) \not\in \{2, +\infty\}$, then $s=1$ is a double zero.
\end{enumerate}
\end{corollary}
\begin{proof}[Proof of Corollary \ref{cor:m2zerosH}.]
By Lemma \ref{lem:zerosH}, $H(s)-s^2$ can have at most $2$ distinct simple zeros inside $\D$ or $1$ zero of order $2$ in $\D$. By Corollary \ref{cor:1val1H}, there is precisely one real negative zero at $s=-\al^{-1}$ of order $1$. Hence, another possible zero of $H(s)-s^2$ must be also of order $=1$, so it must be real positive, because complex zeros of $H(s)-s^2$ with real coefficients should occur in conjugate pairs (see Lemma \ref{lem:divideH}). If such zero exists, then denote it by $s=\be^{-1} \in (0, 1)$. One must have $\be<\al$ by Lemma \ref{lem:rhoH}. It is obvious that $H(s)-s^2$ vanishes at $s=1$. By Lemma \ref{lem:ord1}, it must be of multiplicity $\leq 2$, and derivative calculations in \eqref{eq:lim1} result in conditions for $H'(1)$ and $H''(1)$.
\end{proof}

\section{Asymptotic expansion}\label{sec:asymp_exp}
In this section we decompose $X(s)$ in (\ref{eq:X}) into a simple fractions and prove the main results of the article.

To deal with zeros on the boundary of the disk of convergence of $H(s)$ when decomposing $X(s)$, we need a lemma on the local behavior near the point of singularity. 

\begin{lemma}\label{lem:diffdiv}
Let $k, n \in \N_0$, $\DD \subset \mathbb{C}$ be non-empty convex open set and $\overline{\DD}$ be the closure of $\DD$. Suppose that the function $f: \overline{\DD} \to \mathbb{C}$ is at least $n+k+1$ times continuously complex-differentiable inside the intersection of $\overline{\DD}$ an the open convex neighbourhood $\mathcal{U}$ of a point $\ze \in \partial{\DD}$; here, all the derivatives are taken in such a way, that the variable $s$ approaches $\ze$ while staying in $\overline{\DD}$. If
    \[
        f(\ze)=f'(\ze)=\dots=f^{(n)}(\ze)=0,
    \]
    then $\ze$ is a removable singularity for $q(s):=f(s)/(s-\ze)^{n+1}$ and its derivatives $q^{(j)}(s)$, $0 \leq j \leq k$. Hence, $q^{(j)}(s)$ may be deemed to be continuous at $\mathcal{U} \cap \overline{\DD}$.
\end{lemma}

\begin{proof}[Proof of Lemma \ref{lem:diffdiv}.]
    Let $s \in \mathcal{U} \cap \overline{\DD}$. For $\la \in [0, 1]$, let $s:=s(\la)=\la(s-\ze)+\ze$. Define the function $g: [0, 1] \mapsto \mathbb{C}$ by $g(\la) := f(s) = f(\la(s-\ze)+\ze)$. By the convexity, $s \in \mathcal{U} \cap \overline{\DD}$. Therefore, the complex-valued function $g(\la)$ of a real variable $\la$ is at least $n+k+1$ times real-differentiable in the interval $[0, 1]$ with one-sided derivatives at endpoints. Let $T_n(\la): = \sum_{j=0}^n g^{(j)}(0)\la^j/j!$ be the Taylor polynomial of $g(\la)$ at $\la=0$ of degree $n$, and let $R_n(\la)=g(\la)-s_n(\la)$ be the remainder term. As the Integral Remainder Theorem is applicable to such function as $g(\la)$ (see \cite[Section 12.5.4 in p. 94]{Shilov})
    \[
        R_n(\la) = 1/n!\int_{0}^{\la}g^{(n+1)}(\tau)(\la - \tau)^n\,d\tau,
    \] where $\tau \in [0, \la]$. On the other hand, the Chain Rule differentiation yields
    \[
        g^{(j)}(\tau) = f^{(j)}(s(\tau))(s-\ze)^j, \quad 0 \leq j\leq n+k. 
    \]
    Since the first $n$ derivatives of $f(s)$ vanish at $s=\ze$, we have $g^{(j)}(0)=0$, for $0 \leq j \leq n$. Therefore, $T_n(\la)=0$, $R_n(\la)=g(\la)$ and
    \[
         g(\la) = 1/n!\int_{0}^{\la}f^{(n+1)}(s(\tau))(s-\ze)^{n+1}(\la - \tau)^n\,d\tau   
    \]
    or
    \[
        g(\la)/(s-\ze)^{n+1} = 1/n!\int_{0}^{\la}f^{(n+1)}(s(\tau))(\la - \tau)^n\,d\tau.
    \]
    Setting $\la=1$, one obtains
    \[
        f(s)/(s-\ze)^{n+1} = 1/n!\int_{0}^{1}f^{(n+1)}(\tau(s-\ze)+\ze))(1 - \tau)^n\,d\tau.
    \]
    Hence,
    \[
        \lim_{s \to \ze} f(s)/(s-\ze)^{n+1} =f^{(n+1)}(\ze)/(n+1)!,
    \]
    as long as $s \in \overline{\DD}$.  Thus, $s=\ze$ is a removable singularity. Moreover, the above integrand is $k$ times continuously differentiable in $\mathcal{U} \cap \overline{\DD}$ with respect to the parameter $s$. Therefore, by repeatedly differentiating under the integral with respect to $s$ according to the Leibniz integral rule (an adaptation of \cite[Ex.2, Ch.4]{Conway} with an endpoint on $\pD$) and then taking the limit as $s \to \ze$ with $s \in \mathcal{U} \cap \overline{\DD}$, the function $f(s)/(s-\ze)^{n+1}$ is seen to be $k$ times continuously differentiable in $\mathcal{U} \cap \overline{\DD}$.  
\end{proof}

\begin{theorem}\label{thm:decompX}
    Assume that the $H(s)-s^2$ is primitive. As $H(s)-s^2$ vanishes at $s=1$ with order $r \leq 2$, assume that $H^{(2r+l)}(1) < +\infty$ for some $l \in \N$. Then the generating function $X(s)$ from \eqref{def:X} and \eqref{eq:X} is represented in $\overline{\D}$ by
    \begin{equation}\label{eq:decompX}
        X(s) = \frac{a}{1+\al s} + \frac{b}{1-\be s} 
        + \sum_{j=1}^{r}\frac{c_j}{(1-s)^j} + f(s),
    \end{equation}
    where $-\al^{-1}$ denotes the single negative real zero of $H(s)-s^2$ in $(-1, 0)$, $\be^{-1}$ (if present) denotes the smallest positive real zero of $H(s)-s^2$ in $(0, 1)$, $\al > \be > 1$, and the $l$-th derivative $F^{(l)}(s)$ of the remainder term function $F: \overline{\D} \to \mathbb{C}$ is holomorphic in $\D$ and continuous on $\pD$. The coefficients $a$, $b$, $c_j$, $1 \leq j \leq r$ are real numbers.
\end{theorem}

\begin{proof}[Proof of Theorem \ref{thm:decompX}.]
    Since $H(0)=h_0 \ne 0$ and $s^m \ne 0$ for $s \ne 0$,
    the functions $H(s)$ and $H(s)-s^2$ have no common zero. By Lemma \ref{prop:Hprops} and Corollary \ref{cor:m2zerosH}, $X(s)$ possesses simple poles at the zeros of $H(s)-s^2$: one always occurs at $s=-\al^{-1}$, another occurs at $\be^{-1}$ when $H'(1)>2$. Therefore, $X(s)$ is meromorphic inside $\D$. Since $H^{(2r+l)}(1)< +\infty$, the derivatives $H^{(j)}(s)$, $0 \leq j \leq 2r+l$ are continuous on $\overline{\D}$ by Lemma \ref{prop:Hprops}. As $X(s)$ is a ratio of $H(s)$ and $H(s)-s^2$, it follows that the  derivatives $X^{(j)}(s)$, $0 \leq j \leq 2r+l$ are continuous on $\overline{\D
    }\setminus\{1\}$, since $s=1$ is the single possible vanishing point of $H(s)-s^2$ on $\pD$ according to Corollary \ref{cor:m2zerosH}. In contrast to $-\al^{-1}$ and $\be^{-1}$, the point $s=1$ is not necessarily an isolated singularity of $X(s)$, as $H(s)$ and $X(s)$, in general, might not be continued holomorphically outside $\D$.
    
    To deal with this, rewrite $X(s)$ as $X(s)=H(s)/((s-1)^r G(s))$, where $G(s):=(H(s)-s^2)/(s-1)^r$. Since $s=1$ is a zero of order $r$, $f(s):=H(s)-s^2$ satisfies $f^{(j)}(1)=0$ for $0 \leq j \leq r-1$, as it was shown in \eqref{eq:lim1} of Lemma \ref{lem:ord1}. By applying Lemma \ref{lem:diffdiv} to $f(s)$ with $n=r-1$, $k=r+l$, $\DD=\D$, $\zeta=1$ one finds that, for $0 \leq j \leq r+l$, the $j$--th derivative of $G(s)=f(s)/(s-1)^r$ is continuous in $\overline{\D}$. Then the function $Q(s) := H(s)/G(s)$ has continuous derivatives $Q^{(j)}(s)$ in $\overline{\D}$ of the same order as $G(s)$ does, and it is meromorphic in $\D$.
    
    Now, consider the Taylor polynomial of order $r-1$ of the function $Q(s)$ at $s=1$ and the corresponding remainder
    \[
        T(s) := \sum_{j=0}^{r-1} \frac{Q^{(j)}(1)}{j!}(s-1)^j, \qquad
        R(s) := Q(s) - T(s).
    \]
    Then one can write
    \begin{equation}\label{eq:preDecX}
        X(s) = \frac{Q(s)}{(s-1)^r} = \frac{R(s)}{(s-1)^r} + \frac{T(s)}{(s-1)^r}.
    \end{equation}
    Setting $c_j = (-1)^j Q^{(r-j)}(1)/(r-j)!$, for $1 
    \leq j \leq r$ and $c_j = 0$ for $j >r$, $r \leq 2$ one obtains
     \begin{equation}\label{eq:defc}
        T(s)/(s-1)^r = \sum_{j=1}^2 c_j/(1-s)^j.
    \end{equation}
    As $R(s)$ is a difference of $Q(s)$ and a polynomial, the derivatives $R^{(j)}(s)$, $0 \leq j \leq r+l$ are continuous in $\overline{\D}\setminus\{-\al^{-1}, \be^{-1}\}$. Since $R^{(j)}(1)=0$, for $j=0, 1, \dots, r-1$, Lemma \ref{lem:diffdiv} can be applied to $f(s)=R(s)$, with $n=r-1$, $k=l$. It follows that the $j$-th derivative of $R(s)/(s-1)^r$, $0 \leq j \leq l$ is continuous near $s=1$ in $\overline{\D}$. Thus, each of these derivatives of $R(s)/(s-1)^r$ are continuous in the whole $\overline{\D}\setminus\{-\al^{-1}, \be^{-1}\}$. However, $R(s)/(s-1)^r$ in $\D$ still has a simple poles at $-\al^{-1}$ and at $\be^{-1}$ if $H'(1)>2$. Let
    \begin{equation}\label{eq:defF}
        f(s) := \frac{R(s)}{(s-1)^r} - \frac{a}{1+\al s} - \frac{b}{1-\be s},
    \end{equation}
    where $a/\al$ and $b/\be$ are equal to the residues of $R(s)/(s-1)^r$ at $s=-\al^{-1}$ and $s=\be^{-1}$ respectively. Then, $f(s)$ and its derivatives up to the order $l$ are holomorphic inside $\D$ and continuous in $\overline{\D}$ \cite[Theorem 10.21]{Rudin}. Putting together \eqref{eq:preDecX}, \eqref{eq:defc} and \eqref{eq:defF}, we obtain the decomposition of $X(s)$ \eqref{eq:decompX} in $\overline{\D}$ with all the claimed properties. 
 \end{proof}

\begin{corollary}\label{cor:cfs}
The coefficients $a$, $b$, $c_1$, $c_2$ in Theorem \ref{thm:decompX} have the following expressions:
    \[
        a = \frac{1}{2 +\al H'(-\al^{-1})},
        \qquad b = \begin{cases}
                        0, &\text{ if } H'(1) \leq 2,\\
                        \frac{1}{2-\be H'(\be^{-1})}, &\text{ if } H'(1) > 2.
                    \end{cases}
    \]
    \smallskip
    \noindent If $r=1$, then $H'(1) \ne 2$, and $c_1 = 1/(2-H'(1))$.\\
    \smallskip
    \noindent If $r=2$, then $ H'(1) = 2 \ne H''(1)$, and
    \[
        c_1=\frac{2H^{'''}(1)-12H^{''}(1)+24}{3(H^{''}(1)-2)^2},\quad c_2=\frac{2}{H''(1)-2}.
    \]
\end{corollary}
    
\begin{proof}[Proof of Corollary \ref{cor:cfs}.]
    Since $-\al^{-1}$ is a solution of $H(s)=s^2$, one has $H(-\al^{-1})=\al^{-2}$. From the decomposition \eqref{eq:decompX} of $X(s)$ obtained in Theorem \ref{thm:decompX}, one has
    \[
        a = \lim_{s \to\,-\al^{-1}}(1+\al s)X(s) = \lim_{s \to\, -\al^{-1}}\frac{\al H(s)}{(H(s)-s^2)/(s+\al^{-1})}=
    \]
    \[
 \frac{\al H(-\al^{-1})}{(H(s)-s^2)'/(s+\al^{-1})' \vert_{s=\al^{-1}}}
    = \frac{\al^{-1}}{H'(-\al^{-1})+2\al^{-1}}=\frac{1}{2+\al H(-\al^{-1})}.
    \]
    Replacing $-\al^{-1}$ with $\be^{-1}$ in the above calculation, one finds $b$. For $r=1$, using $1$ in place of $-\al^{-1}$, one obtains $c_1$. For $r=2$, the evaluation of $c_1$ and $c_2$ is slightly more complicated. One has
    \[
        c_{2-j} = \lim_{s \to 1^-} \frac{(-1)^{2-j}}{j!} \big((s-1)^2X(s)\big)^{(j)}, \qquad \text{ for } j = 0, 1.
    \]
    The last limit is evaluated as follows. We write $(s-1)^2 X(s)=H(s)/G_2(s)$, where $G_2(s):=(H(s)-s^2)/(s-1)^2$. By Lemma \ref{lem:diffdiv}, $G_2(s)$ is at least twice continuously differentiable in $\overline{\D}$ according to the assumptions of Theorem \ref{thm:decompX}. This means the above limit evaluation can be replaced by
    \begin{equation}\label{eq:cform}
        c_{2-j} = \frac{(-1)^{2-j}}{j!} \left(\frac{H(s)}{G_2(s)}\right)^{(j)}\bigg\vert_{s=1}
    \end{equation}
    First, we evaluate $H(s)$ and $H'(s)$, using the appropriate assumptions of Theorem \ref{thm:decompX}. Then one finds $G_2(1)$, $G_2'(1)$ by quotient rule, using higher derivatives and Cauchy Middle Value theorem on the real line (or, alternatively, Lemma \ref{lem:diffdiv}) as $s \to 1^-$ to resolve $0/0$ ambiguities. Finally, one differentiates $j$--times the quotient $H(s)/G_2(s)$ and substitutes the previously found values of $H^{(j)}(1)$, $G_2^{(j)}(1)$, in order to evaluate $c_{2-j}$. For instance, $j=0$ in Eq. \eqref{eq:cform} yields
    \[
        c_2 = \frac{(-1)^{2-0}}{0!} \left(\frac{H(s)}{G_2(s)}\right)^{(0)}\bigg\vert_{s=1} = \frac{H(1)}{G_2(1)}=\left(\frac{(H(s)-s^2)^{(2)}}{((s-1)^2)^{(2)}}\Big\vert_{s=1}\right)^{-1}
    \]
    \[
        = \frac{2}{H''(1)-2}.
    \]
    The evaluation of $c_1$ for $r=2$ is similar, but more elaborate, so the technical details are omitted. 
\end{proof}

\begin{theorem}\label{thm:asymtx} Assume that $H(s)-s^2$ is primitive, has vanishing order $r \leq 2$ at $s=1$ and $H^{(2r+l)}(1) < +\infty$ for some $l \in \N_0$. Then, the Taylor coefficients of $X(s)$ have asymptotic expansion
\begin{equation}\label{eq:xexp}
x_n = a(-1)^n \al^n + b\be^n + p_{r-1}(n) + f_n,
\end{equation}
where $a$, $b$, $\al$, $\be$, $f_n$ are as in Theorem \ref{thm:decompX} and Corollary \ref{cor:cfs}, $p_{r-1}(s) \in \mathbb{R}[s]$ is a polynomial of degree $r-1$, and $f_n = o\left(n^{-l}\right)$, as $n \to \infty$.
\end{theorem}

\begin{proof}[Proof of Theorem \ref{thm:asymtx}.]
The power series expansion at $s=0$ of the terms that appear in the decomposition equation \eqref{eq:decompX} of Theorem \ref{thm:decompX} are
\[
\frac{1}{1+\al s}=\sum_{n=0}^\infty (-1)^n\al^n s^n, \qquad \frac{1}{1-\be s}=\sum_{n=0}^\infty \be^n s^n,
\]
\[
\frac{1}{1-s} = \sum_{n=0}^{\infty}s^n, \qquad \frac{1}{(1-s)^2} = \sum_{n=0}^{\infty}(n+1)s^n,
\]
\[
f(s) = \sum_{n=0}^\infty f_n s^n.
\]
Hence,
\[
X(s) = \sum_{n=0}^\infty (-(1)^n a\al^n + b\be^n + p_{r-1}(n)+f_n)s^n, 
\]
where
\begin{equation}\label{eq:poly}
p_{r-1}(n) := c_1 + c_2(n+1) = (c_1+c_2)+c_2n.
\end{equation}
This proves \eqref{eq:xexp}. It remains to estimate the vanishing rate of the coefficients $f_n$. By Theorem \ref{thm:decompX}, the $l$-th derivative $F^{(l)}(s)$ is holomorphic inside $\D$ and continuous in $\overline\D$. The coefficient of $s^{n-l}$, $n \geq l$, in the Taylor series of $F^{(l)}(s)$ at $s=0$ is $n!/(n-l)!f_n$. On the other hand, Cauchy's integral formula \cite[Ch.5, 1.11]{Conway} yields
\begin{equation}\label{eq:CauchiInt}
\frac{n!}{(n-l)!}f_n = \frac{1}{2\pi {\bf i}} \int_{\pD} \frac{F^{(l)}(s)}{s^{n-l+1}}\, ds = \left[\begin{array}{rcl}
s  &=& e^{2\pi {\bf i} \theta}, \quad \\
ds &=& 2\pi {\bf i} e^{2\pi {\bf i}\theta}\,d\theta,\\
\theta &\in& [0, 1]
\end{array}\right]
\end{equation}
\begin{equation}\label{eq:FourierInt}
    = \int_0^1 F^{(l)}\left(e^{2\pi {\bf i}\theta}\right)e^{-2\pi {\bf i} (n-l)\theta}\,d\theta.
\end{equation}
The last integral gives $(n-l)$'th coefficient of the Fourier series for $F^{(l)}\left(e^{2\pi {\bf i}\theta}\right)$. As $F^{(l)}(s)$ is continuous on $\pD$, it follows that $F^{(l)}\left(e^{2\pi {\bf i}\theta}\right) \in C[0, 1]$, and, by Riemann-Lebesgue Lemma \cite[Theorem 2.8, p.13]{Katznelson}, $n!f_n/(n-l)! \to 0$ as $n \to +\infty$. Therefore, $f_j=o\left(n^{-l}\right)$.
\end{proof}

\begin{remark}
If the $f(s)$ admits holomorphic continuation outside the circle of radius $\rho >1$, centered at $s=0$, then $o\left(n^{-l}\right)$ in \eqref{thm:asymtx} can be strengthened to $O\left(\rho^{-n}\right)$. 
\end{remark}

\begin{remark}
The weakest condition that ensures $f_n=o(n^{-l})$, as $n \to +\infty$ is that of $F^{(l-1)}\left(e^{2\pi i\theta}\right)$ being absolutely continuous on $[0, 1]$. However, there seems to be no easy ways to re-cast this condition in terms of the p.g.f. $H(s)$.
\end{remark}

\begin{remark}
If $\abs{F^{(l)}(s)} \leq M < +\infty$ on $\pD$, where $M$ can be estimated numerically, then a slightly weaker estimate $\abs{f_n} \leq M(n-l)!/n!=O\left(n^{-l}\right)$ that follows from \eqref{eq:CauchiInt} and \eqref{eq:FourierInt} might be much more useful for numerical computations.  
\end{remark}

\begin{theorem}\label{thm:asymtD_n}
Let $H(s)-s^2$ be primitive, with $H^{(2r)}(1) < +\infty$. If its vanishing order at $s=1$ is $r$, then, for $n > n_0$,
\[
1 < D_{2n} < D_{2n+2}, \qquad D_{2n+3} < D_{2n+1} <-1, 
\]
\[
 D_{2n} \to +\infty, \qquad D_{2n+1} \to -\infty,
\]
as $n \to +\infty$.
\end{theorem}
\begin{proof}[Proof of Theorem \ref{thm:asymtD_n}.]
If $H'(1) \leq 2$, then $b=0$, $r=1$ or $2$. By Theorem \ref{thm:asymtx}, with $l=0$,
\[
x_n=(-1)^n a \al^n + p_{r-1}(n) + o(1).
\]
 Then,
\begin{align*}
x_n x_{n+2}&= a^2\al^{2n+2} + (-1)^n a \al^np_{r-1}(n+2) + (-1)^{n+2}a \al^{n+2}p_{r-1}(n)+ o(\al^n),\\
x_{n+1}^2 &= a^2\al^{2n+2} +(-1)^{n+1}2a\al^{n+1}p_{r-1}(n+1) + o(\al^n),
\end{align*}
since $\al > 1$.
Therefore,
\begin{align*}
D_n &= h_0\left(x_n x_{n+2} - x_{n+1}^2\right)\\
&=(-1)^n h_0a\al^n\left(p_{r-1}(n+2)+2\al p_{r-1}(n+1)+\al^2 p_{r-1}(n)\right) + o(\al^n). 
\end{align*}
From \eqref{eq:poly}, the leading term of $p_{r-1}(n+2)+2\al p_{r-1}(n+1)+\al^2 p_{r-1}(n)$ is equal to $c_r(1+\al)^2n^{r-1}$.
Therefore
\begin{equation}\label{eq:asympDn1}
D_n \sim (-1)^n h_0 a c_r (1+\al)^2 n^{r-1} \al^n, \text{ as } n \to \infty.
\end{equation}
Likewise, for $H'(1) > 2$, $b \ne 0$, $r=1$, Theorem \ref{thm:asymtx} (with $l=0$) yields
\[
x_n=(-1)^n a \al^n + b\be^n + c + o(1).
\]
Thus,
\[
x_n x_{n+2}= a^2\al^{2n+2} + (-1)^n a b \al^n\be^n(\be^2 + \al^2)+ b^2\be^{2n+2} + o(\al^n),
\]
and
\[
x_{n+1}^2 = a^2\al^{2n+2} +(-1)^{n+1}2ab\al^{n+1}\be^{n+1} + b^2\be^{2n+2} + o(\al^n),
\]
as $\al>\be>1$. Hence,
\[
D_n = h_0\left(x_n x_{n+2} - x_{n+1}^2\right) = (-1)^n h_0 ab\al^n\be^n (\al+\be)^2 + o(\al^n). 
\]
As $\al>1$ and $\be>1$, $D_{2n} \to +\infty$, $D_{2n+1} \to -\infty$, as $n \to \infty$, and,
\[
\lim_{n \to \infty} D_{n+2}/D_n = \begin{cases}
                                    \al^2, & \text{ if } H'(1) \leq 2,\\
                                    \al^2\be^2, & \text{ if } H'(1) > 2.\\
                                \end{cases}
\]
Hence, there exists $n_0 \in \N$, such that, for every $n > n_0$, $D_{2n}\geq1$, $D_{2n+1}\leq-1$ and $\abs{D_{n+2}} > \abs{D_n}$.
\end{proof}

\begin{remark}\label{rem:Dn}
The expression $x_n x_{n+2}-x_{n+1}^2$ from $D_n$ takes part in Aitken's--$\Delta^2$ convergence acceleration method \cite{Aitken, Pomeranz}, while the ratio $D_{n+1}/D_n$ is used as the numerical estimator for the radius of convergence of power series \cite{Mercer}.
\end{remark}

\begin{proof}[Proof of Theorem \ref{thm:main_res}.]
    As derivatives $H^{(j)}(1)$ can be expressed via the moments $\E{Z^i}$, $0 \leq i \leq j \leq 2r$ (and vice versa), like $H'(1)=\E{Z}$, $H''(1)=\E{Z^2}-\E{Z}$, Theorem \ref{thm:main_res} is just a re--statement of Theorem \ref{thm:asymtD_n}. 
\end{proof}

\begin{proof}[Proof of Corollary \ref{cor:ivp}.]
    Let us first consider the case $\mathbb{P}{\left(Z \in 2\N_0+1\right)} = 0$. This means that every r.v. $Z_i$ in $W(n)$ takes only even values with probability $1$. Consequently, for integer $u \geq 1$, $\vphi(0)=\vphi_1(0)$, $\vphi(1)=\vphi(2\cdot1 -1) = \vphi_1(1)$, where $\vphi_1(u)$ denotes the ultimate survival probability of the process $W_1(n)$ described in Section \ref{sec:genfun}. By replacing $Z$ with $Z/2$ in the well known (see, for instance \cite{Asmussen}) ultimate survival probability formula for $W_1(n)$ in $\ka=1$ case, one obtains $\vphi(0)=1-\E{Z}/2$, since $\E{Z/2} < 1$. Then, using recursion \eqref{eq:main} for $\vphi_1(n)$ with $\ka=1$, one obtains $\vphi(1)$, as claimed.
    
    Let us now consider the case $\mathbb{P}{\left(Z \in 2\N_0+1\right)} \ne 0$. Then $H(s)-s^2$ must be primitive. For the primitive case, $H'(1)=\E{Z} < 2$ results in the vanishing order $r=1$ at $s=1$ for $H(s)-s^2$, and $EZ^2 < +\infty$ yields $H''(1) < +\infty$. Therefore, Theorem \ref{thm:asymtx} and Theorem \ref{thm:asymtD_n} are applicable (with $r=1$ and $l=0$). Then
    \[
    x_n \sim (-1)^n a \al^n, \qquad D_n \sim (-1)^n h_0 a c_1(1+\al)^2\al^{n},
    \] as $n \to \infty$ by Theorem \ref{thm:asymtx} and by Eq. \eqref{eq:asympDn1} in the proof of Theorem \ref{thm:asymtD_n}. Since $y_n=h_0x_{n+1}$ holds by Eq. \eqref{eq:yfromx}, the limits of ratios $x_n/D_n$, $x_{n+1}/D_n$, $y_n/D_n$, $y_{n+1}/D_n$ are
    \[
        \frac{1}{h_0c_1(1+\al)^2}, \quad -\frac{\al}{h_0c_1(1+\al)^2}, \quad -\frac{\al}{c_1(1+\al)^2}, \quad \frac{\al^2}{c_1(1+\al)^2},
    \] respectively. Substituting these limits into expressions in \eqref{eq:ivp} and using $c_1=1/(2-H'(1))=1/(2-\E{Z})$ from Corollary \ref{cor:cfs}, $\vphi(\infty)=1$ for $\E{Z}<2$, one obtains the claimed formulas for $\vphi(0)$ and $\vphi(1)$.
\end{proof}

\section{Generating function for survival probabilities}\label{sec:alt}

Let $S_0 := 0$, $S_n = \sum_{i=1}^n(Z_i - \ka)$ for $n \geq 1$, and set $M := \max_{n\geq 1} S_n$. Recall that the positive part of a real number $a^+ := \max\{0, a\}$. We recall the classical stationarity property for the distribution of the maximum of a reflected random walk \cite[Ch. VI, sec. 9]{Feller}:

\begin{lemma}\label{lem:stat}
The r.v.s. $(M+Z-\ka)^+$ and $M^+$ are distributed identically.
\end{lemma}
\begin{proof}[Proof of Lemma \ref{lem:stat}.]
One has
\begin{align*}
&(M^+ +Z-\ka)^+ \dist \max\{0, M^+ +Z-\ka\} \dist \max\{0, \max\{0, \max_{n \geq 1} S_n\}+Z-\ka\}\\
&\dist\max\{0, \max\{Z-\ka, \max_{n \geq 1}(S_n+Z-\ka)\}\} \dist \max\{0, \max_{n \geq 1}S_n\} \dist M^+.
\end{align*}
\end{proof}

One has $W(n) = u + \ka n - \sum_{i=1}^n Z_i = u-S_n$, and
\[
\vphi(u) = \mathbb{P}{\left(\bigcap_{n=1}^\infty \{W(n) > 0\}\right)} = \mathbb{P}{\left(\bigcap_{n=1}^\infty \{S(n)<u\}\right)} = \mathbb{P}{\left( M < u \right)}.
\]
For $i \in \N_0$, let $\pi_i := \mathbb{P}{(M^+=i)}$. Then $\vphi(u+1)=\sum_{i=0}^{u}\pi_i,\,u\in\mathbb{N}_0$ and, in particular, $\vphi(1)=\pi_0$. Recall that $H(s)=\E{s^Z}$ is the p.g.f. of $Z$ and denote the p.g.f. of $M^+$ by $G(s)=\E{s^{M^+}}$. 

\begin{proof}[Proof of Theorem \ref{thm:Phi}.]
Since $\E{Z^2} < +\infty$, $\E{M^+} < +\infty$, see \cite[Theorems 5 and 6]{Kiefer}. As $(M+Z-2)^+$ and $M^+$ are identically distributed, we have $\E{(M^+ +Z-2)}=\E{M^+}$, or
\[
\E{\left( (M^++Z-2)^+ - M^+\right)} = 0.
\]
On the other hand, by the Law of Total Expectation,
\begin{align*}
&\E{\left((M^++Z-2)^+ - M^+\right)} = \E{\left(\E{((M^++Z-2)^+ - M^+)}\mid M^+\right)}\\
&= \pi_0\E{(Z-2)^+} + \pi_1(\E{(Z-1)^+} -1) + (1-\pi_0-\pi_1)\E{\left(Z-2\right)}\\
&= \pi_0(\E{(Z-2)}+2h_0+h_1) + \pi_1\left(\E{\left(Z-2\right)}+h_0\right) + (1-\pi_0-\pi_1)\E{\left(Z-2\right)}\\
&=\E{(Z-2)}+ \pi_0(2h_0 + h_1) + \pi_1h_0.
\end{align*}
Therefore,
\begin{equation}\label{eq:cond1}
    (2h_0 + h_1)\pi_0 + h_0\pi_1 = 2 - \E{Z}.
\end{equation}
Similarly,
\[
\E{s^{(M^+ +Z-2)^+}} = \E{s^{M^+}} \text{ or } \E{\left(s^{M^+ +Z-2}-s^{M^+}\right)}=0.
\]
By conditioning on $M^+$ again, we obtain
\begin{align*}
&\E{\left(s^{M^+ +Z-2}-s^{M^+}\right)} = \E{\left(\E{\left(s^{M^+ +Z-2}-s^{M^+}\right)} \mid M^+\right)}\\
&= \pi_0\E{\left(s^{(Z-2)^+}- 1\right)} + \pi_1\E{\left(s^{(Z-1)^+} - s\right)} + \sum_{i=2}^\infty \pi_i \E{\left(s^{i+Z-2}-s^i\right)}\\
&=\pi_0\E{\left(s^{Z-2} + h_0(1-s^{-2}) + h_1(1-s^{-1}) - 1\right)}\\
&+ \pi_1\E{\left(s^{Z-1} +h_0(1-s^{-1}) - s\right)} +  \sum_{i=2}^{\infty} \pi_i \E{\left(s^{i+Z-2}-s^i\right)}\\
&= H(s)\left(\pi_0s^{-2} + \pi_1s^{-1} + \sum_{i=2}^\infty\pi_i s^{i-2}\right) - \sum_{i=2}^{\infty} \pi_i s^i\\
&+ \pi_0(h_0(1-s^{-2}) + h_1(1-s^{-1}) - 1) + \pi_1(h_0(1-s^{-1}) - s)\\
&=s^{-2}H(s)\left(\sum_{i=0}^\infty\pi_i s^i\right) - \sum_{i=0}^{\infty}\pi_is^i + (1-s^{-1})\left(\pi_0h_0(1+s^{-1}) + \pi_0h_1 + \pi_1h_0\right)\\
&=  G(s)(s^{-2}H(s)-1) + (1-s^{-1})\left(\pi_0h_0(1+s^{-1}) + \pi_0h_1 + \pi_1h_0\right)=0
\end{align*}
or
\begin{equation}\label{eq:lygtis}
    G(s)(H(s)-s^2) = (1-s)\left(\pi_0h_0 + (\pi_0h_0+\pi_0h_1 + \pi_1h_0)s\right).
\end{equation}
As $G(s)$ converges absolutely in $\overline{\D}$, one is allowed to evaluate  \eqref{eq:lygtis} at $s=-\al^{-1}$, where $H(s)=s^2$. In doing so one obtains the second linear equation for $\pi_0$ and $\pi_1$
\begin{equation}\label{eq:cond2}
\pi_0(h_0+h_1-h_0\al) + \pi_1h_0=0.
\end{equation}

Combining the last equation with \eqref{eq:cond1}, one obtains, for $\E{Z} < 2$
\[
\pi_0 = \vphi(1) = \frac{2-\E{Z}}{h_0(1+\al)}, \qquad \pi_1 = \frac{(2-\E{Z})(\al-1 - h_1/h_0)}{h_0(1+\al)}.
\]
By \eqref{eq:lygtis} and the last expressions of $\pi_0$ and $\pi_1$,
\[
G(s) = \frac{(2-\E{Z})(1-s)(1+\al s)}{(1+\al)(H(s)-s^2)}.
\]
Then, the g.f. $\Xi(s)=\sum_{i=0}^{\infty}\vphi(i+1)s^i$ of probabilities $\vphi(u+1),\,u\in\mathbb{N}_0$ is 
\begin{equation}\label{eq:lygtis2}
\Xi(s) = \frac{G(s)}{1-s} = \frac{(2-\E{Z})(1+\al s)}{(1+\al)(H(s)-s^2)}.
\end{equation}
It should be noted that \eqref{eq:lygtis2} is valid only for $\E{Z} \leq 2$: for $\E{Z} > 2$, the right hand-side of \eqref{eq:cond1} becomes negative, which means $\pi_0$ or $\pi_1$ is negative; thus, $\Xi(s)$ could not be a g.f. of probabilities $\vphi(u+1),\,u\in\mathbb{N}_0$. For $\E{Z} > 2$, $\Xi(s)=0$, which we incorporate into \eqref{eq:lygtis2} by using $(2-\E{Z})^+$. As $\Xi(s)$  does not contain $\vphi(0)$, it must be solved from the recursion \eqref{eq:main}: $\vphi(0)=h_0\varphi(2)+h_1\varphi(1)= h_0(\pi_0 + \pi_1) + h_1\pi_0$ which confirms the result of Corollary \ref{cor:ivp}.
\end{proof}

\section{Some specific examples}\label{sec:examples}
\subsection{Bernoulli's distribution.} Let $Z \sim \BB(p)$ denote the Bernoulli r.v. with success probability $0 < p < 1$. Then its p.g.f. $H(s)=q+ps$, where $q=1-p$. In this case 
\begin{align*}
X(s)=\frac{q+ps}{q+ps-s^2}=\frac{1}{(1+q) (1-s)} + \frac{q}{(1 + q) (1+s/q)}.
\end{align*}
Then 
\[
x_n=\frac{1+(-1)^{n}q^{1-n}}{1+q}, \quad n \in \N_0.
\]
The determinant $D_n$ evaluates to
\begin{align*}
 D_n &= h_0(x_{n+2}x_n-x_{n+1}^2)\\
     &= \frac{q}{(1+q)^2}\left(\left(1+\frac{(-1)^{n}}{q^{-1+n}}\right)\left(1+\frac{(-1)^{n+2}}{q^{1+n}}\right)- \left(1+\frac{(-1)^{n+1}}{q^{n}}\right)^2\right)\\
     &= \frac{q}{(1+q)^2}\left(1+\frac{(-1)^{n}}{q^{-1+n}} + \frac{(-1)^{n}}{q^{1+n}} + \frac{1}{q^{2n}} - 1 -\frac{2(-1)^{n+1}}{q^{n}}-\frac{1}{q^{2n}}\right)\\
     &=\frac{(-1)^n q}{(1+q)^2}\left(\frac{(-1)^{n}}{q^{-1+n}} + \frac{(-1)^{n}}{q^{1+n}}-\frac{2(-1)^{n+1}}{q^{n}}\right)\\
     &=\frac{(-1)^n q^{1-n}}{(1+q)^2}\left(q + q^{-1} + 2\right)=\frac{(-1)^n q^{1-n}}{(1+q)^2}\frac{(1+q)^2}{q}= \frac{(-1)^n}{q^n}.
\end{align*}
Therefore, for $Z \sim \BB(p)$, Conjecture \ref{conj:Dn} is true.

\subsection{Geometric distribution}

Let $Z \sim \GG(p)$ denote the geometric distribution 
$\mathbb{P}(Z=k)=p(1-p)^k$, $k=0,\,1,\,\ldots$ with a p.g.f.
\[
H(s)=\frac{p}{1-qs}, \qquad 0 < p,\, q <  1, \; p+q=1
\]
and
\[
X(s)=\frac{1}{(1-s)(1+s-qp^{-1}s^2)}=\frac{1}{(1-s)(1+\al s)(1-\be s)}
\]
where
\[
\al = (\sqrt{4p^{-1}-3}+1)/2, \qquad \be = 
(\sqrt{4p^{-1}-3}-1)/2, 
\]
satisfy

\begin{equation}\label{eq:beta}
\be > 1, \text{ for } 0 < p < 1/3, \qquad 0 < \be < 1, \text{  for } 1/3 < p < 1,
\end{equation}
\[
\al > \max\{1, \be\}, \text{ for } 0 < p < 1.
\]

For $p \ne 1/3$, $X(s)$ decomposes into
\[
X(s) = \frac{a}{1+\al s} + \frac{b}{1-\be s} + \frac{c_1}{1-s},
\]
where
\[
a = \frac{q+\al}{3q+2\al}, \qquad b = \frac{q-\be}{3q-2\be}, \qquad c_1 =\frac{p}{3p-1}, 
\]
(see Corollary \ref{cor:cfs}) satisfy
\begin{align}\label{eq:bc}
    1/2 < b < +\infty,& \text{ if } p \in (0, 1/3), &  -\infty < b <0,& \text{ if } p \in (1/3, 1),\\
 0 > c_1 > -\infty,& \text{ if } p \in (0, 1/3), &  + \infty > c_1 > 1/2,& \text{ if } p \in (1/3, 1),\nonumber
\end{align}
\[
4/9 < a < 1/2, \text{ for } 0 < p < 1.
\]

It follows that
\[
x_n = (-1)^n a\al^n + b\be^n + c_1,
\]
and
\begin{align*}
&\frac{(-1)^n}{h_0} D_n =(-1)^n(x_n x_{n+2}-x_{n+1}^2)\\
&= ab(\al + \be)^2\al^n\be^n + (-1)^nbc_1(\be-1)^2\be^n + c_1 a(1 + \al)^2\al^n. 
\end{align*}

Consider
\begin{align*}
&(-1)^n/h_0\left(D_{n+2}-\al^2D_n\right)\\
&=ab(\al + \be)^2(\be^2-1)\al^{n+2}\be^n + (-1)^n bc_1(\be-1)^2(\be^2-\al^2)\be^n\\ 
&=ab(\al + \be)^2(\be^2-1)\be^n\left(\al^{n+2} + (-1)^n \frac{c_1(1-\be)(\al-\be)}{a(\al + \be)(\be+1)} \right)\\
&=ab(\al + \be)^2(\be^2-1)\be^n\left(\al^{n+2} + (-1)^n f(p) \right),
\end{align*}
where
\[
f(p) := \frac{c_1(1-\be)(\al-\be)}{a(1+\be)(\al + \be)}
=\frac{p^2(\al+p-2)}{(1-p)^3}.
\]
The last expression for $f(p)$ was found with \texttt{Mathematica} \cite{Mathematica} and verified with \texttt{Sage} \cite{Sage}. Since $a>0$ and the signs of $\be^2-1$ and $b$ in \eqref{eq:beta}, \eqref{eq:bc} coincide for $p \in (0, 1) \setminus \{1/3\}$, it follows that the sign of $(-1)^n(D_{n+2}-\al^2D_n)$ matches the sign of $\al^{n+2}+(-1)^nf(p)$. 

For $f(p)$ it holds that $\lim_{p\to 0^+ }f(p)=0$, $\lim_{p\to 1^-} f(p)=1$ and $f'(p)>0$ for $p\in(0,1)$, see Figure \ref{fig:pic}. Therefore, $f(p) < 1$ and $\al^{n+2} + (-1)^n f(p) > 0$ for $p \in (0, 1)$ due to $\al > 1$. Hence, for every $n \in \N_0$, it holds that $D_{2n+3} < \al^2 D_{2n+1}< D_{2n+1} \leq -1$ and $D_{2n+2} > \al^2 D_{2n+1}> D_{2n+1} \geq 1$. 

\begin{figure}[H]
    \centering
    \includegraphics[scale=0.33]{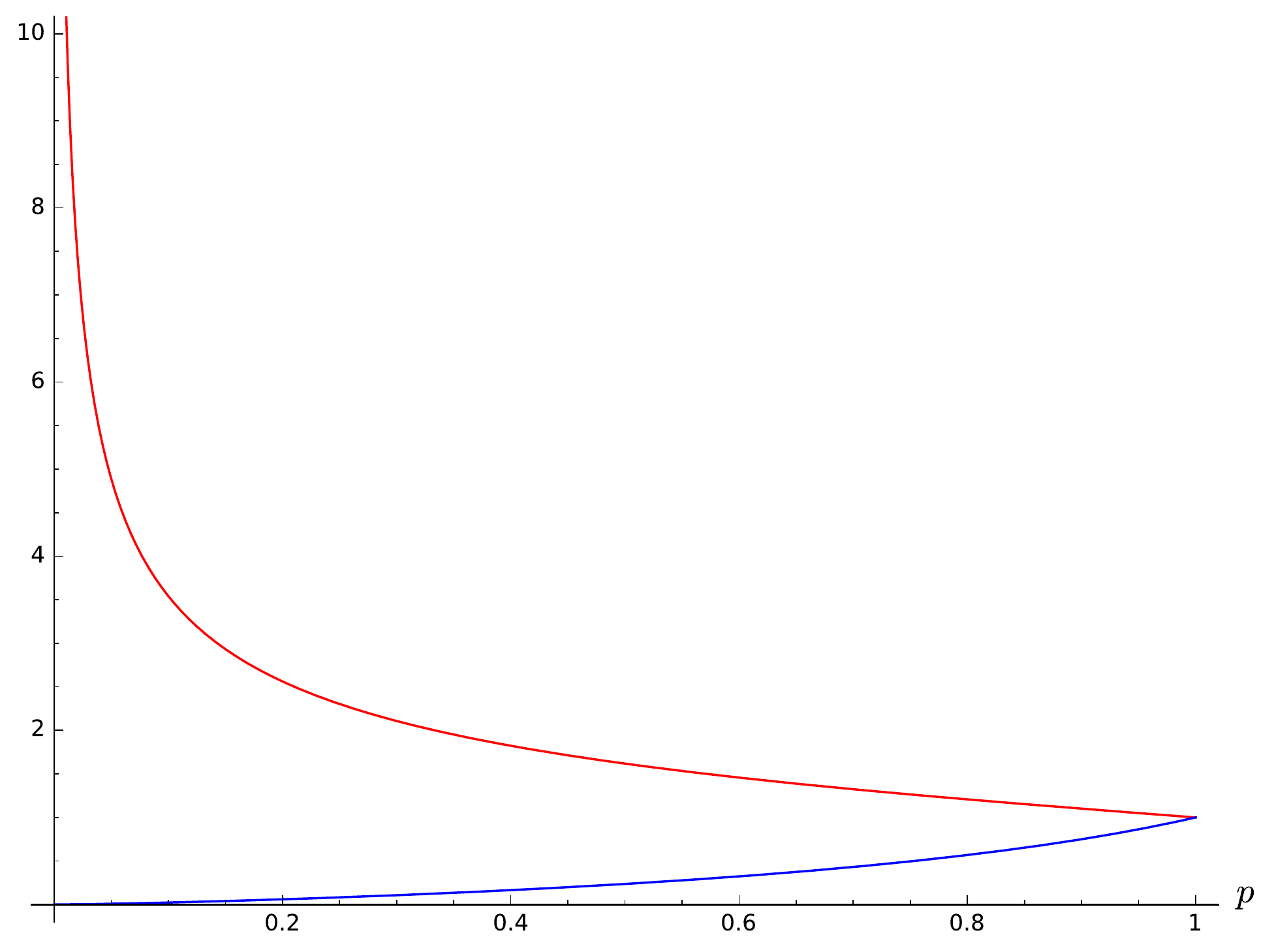}
    \caption{$\al=\al(p)$ (red) v.s. $f(p)$ (blue), $p \in (0, 1)$}
    \label{fig:pic}
\end{figure}

It remains to consider the value $p=1/3$. In this case,
\[
X(s)=\frac{1}{(1-s)(1+s-2s^2)}=\frac{1}{(1+2s)(1-s)^2}
\]
\[
=\frac{4/9}{1+2s} + \frac{2/9}{1-s} + \frac{1/3}{(1-s)^2}.
\]

This yields
\[
x_n = ((-2)^{n+2}+5+3n)/9, \qquad D_n = ((-2)^{n+2}(27n+63)-9)/81,
\]
\[
(-1)^n(D_{n+2}-D_n) = 2^{n+2}(n+5) > 0.
\]
In conclusion, Conjecture \ref{conj:Dn} is true for $Z \sim \GG(p)$. 

\section{Acknowledgments}\label{sec:ackn}

We thank anonymous reviewer for pointing out the derivation of the generating function $\Xi(s)$ for survival probabilities using the stationarity property for the maximum of the positive parts of sums of random variables which is outlined in Section \ref{sec:alt}. We also appreciate constructive criticism and revealed misprints/inaccuracies of another anonymous referee that helped to improve paper's overall quality and readability. Last but not least, we want to thank to professor Jonas Šiaulys for his comments on the manuscript.


\begin{thebibliography}{20}

\bibitem{Aitken} Aitken, A.: On Bernoulli's numerical solution of algebraic equations. Proceedings of the Royal Society of Edinburgh. {\bf 46}, 289--305 (1927). \url{http://dx.doi.org/10.1080/23311835.2017.1308622}

\bibitem{Andersen} Andersen, E.S.: On the collective theory of risk in case of contagion between the claims. Trans. Xvth Int. Actuar. {\bf 2}, 219--229 (1957)

\bibitem{Asmussen} Asmussen, S., Albrecher, H.: Ruin Probabilities. World Scientific: Singapore, (2010). \url{https://doi.org/10.1142/7431}

\bibitem{Conway} Conway, J.B.: Functions of one complex variable, 2nd ed. Springer--Verlag, New York, (1978)

\bibitem{Curt} Curtiss, D.R.: Recent Extensions of Descartes' Rule of Signs. Annals of Mathematics. {\bf 19} (4), 251--278  (1918)

\bibitem{DS} Damarackas, J., Šiaulys, J.: Bi-seasonal discrete time risk model, Appl. Math. Comput. {\bf 247}, 930--940 (2014). \url{https://doi.org/10.1016/j.amc.2014.09.040}

\bibitem{Dickson_Waters} Dickson, D.C.M., R. Waters, R.:, Recursive calculation of survival probabilities. ASTIN Bull. {\bf 21}, 199--221, (1991). \url{https://doi.org/10.2143/AST.21.2.2005364}

\bibitem{Feller} Feller, W.: An Introduction to Probability Theory and Its Applications. Vol. 2, 2nd ed., Wiley, New York, (1971)

\bibitem{Gantmacher} Gantmacher, F.R.: The theory of matrices. Vol. 1. translated by K. A. Hirsch, reprint of the 1959 translation. AMS Chelsea Publishing, Providence, RI, (1998)

\bibitem{Graham}Graham, E., van der Poorten, A., Shparlinsky I., Ward, T.: Recurrence sequences, Mathematical Surveys and Monographs, vol. 104, American Mathematical Society, Providence, RI, (2003). \url{https://doi.org/10.1090/surv/104}

\bibitem{Gerber}Gerber, H.U.: Mathematical fun with the compound binomial process. ASTIN Bull., {\bf 18}, 161--168, (1988). \url{https://doi.org/10.2143/AST.18.2.2014949}

\bibitem{Gerber1} Gerber, H.U.:, Mathematical fun with ruin theory. Insur. Math. Econ. {\bf 7}, 15--23, (1988). \url{https://doi.org/10.1016/0167-6687(88)90091-1}

\bibitem{GKS} Grigutis, A., Korvel, A., Šiaulys, J.: Ruin probability in the three-seasonal discrete-sime risk model, Mod. Stochastics: Theory Appl., {\bf 2}, 421--441, (2015). \url{https://doi.org/10.15559/15-VMSTA45}

\bibitem{GS1} Grigutis, A., Šiaulys, J.: Ultimate Time Survival Probability in Three-Risk Discrete Time Risk Model, Mathematics, {\bf 8} (2), 147--176, (2020). \url{https://doi.org/10.3390/math8020147}

\bibitem{GS} Grigutis, A., Šiaulys, J.: Recurrent Sequences Play for Survival Probability of Discrete Time Risk Model, Symmetry, {\bf 12} (12), 2111--2131, (2020). \url{https://doi.org/10.3390/sym12122111}

\bibitem{Katznelson}Katznelson, Y.: Introduction to Harmonic Analysis, 3rd Edition, Cambridge University Press, California, (2004). \url{https://doi.org/10.1017/CBO9781139165372}

\bibitem{Kiefer} Kiefer, J., Wolfowitz, J.: On the Characteristics of the General Queueing Process, with Applications to Random Walk, The Annals of Mathematical Statistics, {\bf 27}(1), 147–61, (1956). http://www.jstor.org/stable/2236981.

\bibitem{Mathematica} Mathematica (Version 9.0), Wolfram Research, Inc., Champaign, Illinois, 2012, \url{https://www.wolfram.com/mathematica}

\bibitem{Mercer} Mercer, G.N., Roberts, A.J.: A centre manifold description of contaminant dispersion in channels with varying flow properties. SIAM Journal on Applied Mathematics {\bf 50} (6), 1547--1565, (1990). \url{http://www.jstor.org/stable/2101904}

\bibitem{Pomeranz}Pomeranz, S.B.: Aitken’s $\Delta^2$ method extended, Cogent Mathematics, {\bf 4} (1) (2017). \url{http://dx.doi.org/10.1080/23311835.2017.1308622}

\bibitem{Rudin}Rudin, W.: Real and complex analysis, 3rd ed. McGraw-Hill, New York, (1987)

\bibitem{Sage} SageMath, the Sage Mathematics Software System (Version 8.1), The Sage Developers, 2017, \url{https://www.sagemath.org}

\bibitem{Shilov}Shilov, G.E.: Mathematical Analysis. Part 3: Functions in One Variable (in Russian). Nauka, Moscow, (1973)

\bibitem{Shiu} Shiu, E.S.W.: Calculation of the probability of eventual ruin by Beekman's convolution
series, Insur.~Math.~Econ., {\bf 7}, 41--47, (1988). \url{https://doi.org/10.1016/0167-6687(88)90095-9}

\end{thebibliography}

\end{document}